\documentclass[reqno,11pt]{amsart}
\usepackage{a4wide,color,eucal,enumerate,mathrsfs}
\usepackage[normalem]{ulem}
\usepackage{amsmath,amssymb,epsfig,amsthm} 
\usepackage[latin1]{inputenc}
\usepackage{psfrag}
\numberwithin{equation}{section}

\newtheorem{thm}{Theorem}[section]
\newtheorem{lem}[thm]{Lemma}

\numberwithin{equation}{section}

\theoremstyle{remark}

\newcommand{\R}{\mathbb{R}}

\def\rr{{\mathbb R}}

\def\zz{{\mathbb Z}}

\def\bint{{\ifinner\rlap{\bf\kern.35em--}
\int\else\rlap{\bf\kern.45em--}\int\fi}\ignorespaces}

\def\bbint{{\ifinner\rlap{\bf\kern.35em--}
\hspace{0.078cm}\int\else\rlap{\bf\kern.45em--}\int\fi}\ignorespaces}

\def\bint{{\ifinner\rlap{\bf\kern.35em--}
\int\else\rlap{\bf\kern.45em--}\int\fi}\ignorespaces}

\begin{document}

\title[Weighted fractional Sobolev-Poincar\'e inequalities in irregular domains]
{Weighted fractional Sobolev-Poincare inequalities in irregular domains}
\author{Yi Xuan\textsuperscript{1,2}}
\address{$1.$ HLM, Academy of Mathematics and Systems Science,
Chinese Academy of Sciences, Beijing, 100190, People's Republic of China}
\address{$2.$ School of
Mathematical Sciences, University of Chinese Academy of Sciences,
Beijing, 100049, People's Republic of China}
\email{xuanyi@amss.ac.cn}

\date{\today}


\begin{abstract}
In this paper, we study weighted fractional Sobolev-Poincar\'e inequalities for irregular domains. The weights considered here are distances to the boundary to certain powers, and the domains are the so-called $s$-John domains and $\beta$-H\"older domains. Our main results extend that of Hajlasz-Koskela [J. Lond. Math. Soc. 1998] from the classical weighted Sobolev-Poincar\'e inequality to its fractional counter-part and Guo [Chin. Ann. Math. 2017] from the frational Sobolev-Poincar\'e inequality to its weighted case.
\end{abstract}


\maketitle

\section{Introduction}

The classical Sobolev-Poincar\'e inequality asserts that if $u$ is a smooth function on $\Omega\subset \R^n$ and $\Omega$ is a bounded Lipschitz domain, then for all $p\in [1,n)$,
\begin{equation}\label{eq:class Sob-Poin}
	\Big(\int_{\Omega}|u-u_\Omega|^{\frac{np}{n-p}}dx \Big)^{\frac{n-p}{np}}\leq C\Big(\int_{\Omega}|\nabla u|^pdx \Big)^{\frac{1}{p}},
\end{equation}
where $u_{\Omega}=|\Omega|^{-1}\int_{\Omega}udx$ is the integral average of $u$ in $\Omega$. This inequality is closely related to the Rellich-Kondrachov compactness embedding theorem, which has important applications in the theory of partial differential equations, see for instance \cite{B2011,M2011}. The validity of \eqref{eq:class Sob-Poin} in irregular domains has also gained a lot of interest since the 1990s, partially because of its application in variational problems, see \cite{SS1990,BK1998,HK1998,HK2000,KM2000} and the references therein. In particular, weighted versions of \eqref{eq:class Sob-Poin} have been established in \cite{HK1998} together with applications in the compact embedding problem. One of the main technical innovation in \cite{HK1998} is to show the validity of (weighted) Sobolev-Poincar\'e inequality is equivalent with certain capacity estimates. This useful  observation was already used in the earlier works of Maz'ya \cite{M1960,M1973}, but only for bounded domains with the cone condition. In \cite{HK1998}, similar estimates were extended to the more general class of $s$-John domains.

Recall that a bounded domain $\Omega$ in $\rr^n$ ($n\geq 2$) is called an $s$-John domain ($s\geq 1$) if there exists a constant $C$ and a distinguished point $x_0$ satisfying that, for each $x\in\Omega$ there is an arc-length parametrization rectifiable curve $\gamma:[0,\ell(\gamma)]\rightarrow \Omega$, where $\ell(\gamma)$ is the length of $\gamma$, with $\gamma(0)=x$ and $\gamma(\ell(\gamma))=x_0$ such that for all $0\leq t\leq\ell(\gamma)$, it holds
\[
Cd(\gamma(t),\partial\Omega)\geq t^s.
\]
In the case $s=1$, this concept was first used by F. John in his work on elasticity \cite{J1961} and the term was introduced by Martin and Sarvas\cite{MS1979}. For general $s\geq 1$, it was introduced by Smith and Stegenga \cite{SS1990}; see also \cite{GK14} for further extensions of this class of domains and its connection with geometric function theory.


Recently, there has been growing interest in the study of the fractional Sobolev-Poincar\'e inequality; see for instance \cite{DIV2015,HSV2013,G2017,DD2022}. More precisely, the following fractional $(p,q)$-Sobolev-Poincar\'e inequality in a domain $\Omega\subset\rr^n$ ($n\geq2$) with finite Lebesgue measure were largely considered in the literature:
\begin{equation}\label{eq:fract Sob Poin}
\int_{\Omega}{|u(x)-u_{\Omega}|}^qdx\leq C\Big(\int_{\Omega}\int_{\Omega\cap B(x,\tau d(x,\partial \Omega))} \frac{{|u(x)-u(y)|}^p}{{|x-y|}^{n+p\delta}}dydx\Big)^{\frac{q}{p}},
\end{equation}
where $1\leq p\leq q<\infty$, $\delta\in (0,1)$, $\tau\in(0,\infty)$ and the constant $C$ is independent of $u\in C(\Omega)$. To simplify our notation, for each $u\in C(\Omega)$, $\delta\in (0,1)$, $\tau\in(0,\infty)$, we define a function $g_u^{\tau}$ on $\Omega$ as
\[
g_u^\tau(x)=\int_{\Omega\cap B(x,\tau d(x,\partial\Omega))}\frac{{|u(x)-u(y)|}^p}{{|x-y|}^{n+p\delta}}dy.
\]
Based on the idea of Hajlasz-Koskela \cite{HK1998}, in \cite{G2017}, an essentially sharp version of \eqref{eq:fract Sob Poin} was shown to hold in the class of $s$-John domains.

The aim of present paper is to give a further extension of the main results of \cite{G2017} to the weighted case, similar as in \cite{HK1998}. To be more precise, let $f,g$ be two positive continuous functions defined on an open set $\Omega\subset \R^n$ with $\int_{\Omega}f(x)dx<\infty$. In this paper, we are interested in the following weighted version of \eqref{eq:fract Sob Poin}:
\begin{equation}\label{eq:weighted fractional Sob Poin}
	{\Big(\int_\Omega{|u(x)-u_{\Omega,f}|}^{q}f(x)dx\Big)}^{\frac{1}{q}}\leq C {\Big(\int_\Omega g_u^\tau(x)g(x)dx\Big)}^{\frac{1}{p}},
\end{equation}
where $u_{\Omega,f}$ is the $f$-average of $u$ on $\Omega$ defined by
\[u_{\Omega,f}=\frac{\int_\Omega u(x)f(x)dx}{\int_\Omega f(x)dx}.\]
When $f\equiv g\equiv 1$, \eqref{eq:weighted fractional Sob Poin} reduces to \eqref{eq:fract Sob Poin}.



The general idea towards \eqref{eq:weighted fractional Sob Poin} is similar to \cite{HK1998} and \cite{G2017}. The starting point is to reduces the validity of the weighted fractional Sobolev-Poincar\'e inequality \eqref{eq:weighted fractional Sob Poin} to certain weighted capacity estimates. To formulate our main result, let us first recall that a set $A$ is called an admissible subset of an open set $\Omega\subset \R^n$ if $A$ is an open set and $\partial A\cap\Omega$ is a smooth submanifold. Our main result of this paper reads as follows.

\begin{thm}\label{thm:capactity characterization}
Let $f,g$ be two positive continuous functions on an open set $\Omega\subset \R^n$ with $\int_{\Omega}f(x)dx<\infty$. Let $1\leq p \leq q<\infty$, $\delta\in (0,1)$ and $\tau\in(0,\infty)$. Assume that for any ball $B\subset\subset\Omega$, there exists a constant  $C=C(\Omega,B,f,g,p,q,\delta,\tau)$ such that
\[
\Big(\int_Af(x)dx\Big)^{\frac{p}{q}}\leq C \inf_{u}\int_\Omega g_u^\tau(x)g(x)dx,
\]
for any admissible set $A\subset\Omega$ with $A\cap B=\emptyset$, where the infimum is taken with respect to all the functions $u\in C(\Omega)$, satisfying that $u|_A\geq 1$ and $u|_B=0$. Then there exists a constants $\tau_0=\tau(\Omega)$ such that for each $\tau\in (0,\tau_0)$ and each $u\in C(\Omega)$, the weighted fractional Sobolev-Poincar\'e inequality \eqref{eq:weighted fractional Sob Poin} holds.
\end{thm}

Theorem \ref{thm:capactity characterization} can be regarded as a weighted fractional version of \cite[Theorem 1]{HK1998} and \cite[Theorem 1.1]{G2017} and it allows us to study the weighted fractional Sobolev-Poincar\'e inequality \eqref{eq:weighted fractional Sob Poin} in irregular domains via capacity estiamtes. The techniques for doing capacity estiamtes in $s$-John domians are now well developed; see for instance \cite{HK2000,G2015}.

Similar as in \cite{HK1998}, for a point $x\in\Omega$, we set $\rho(x)=d(x,\partial \Omega)$. Our second main result can be regarded as a weighted version of \cite[Theorem 1.2]{G2017}.


%

\begin{thm}\label{4.4t}
Let $\Omega\subset \R^n$ be an $s$-John domain. Assume that $\delta\in(0,1)$, $1\leq p\leq q$, $a\geq 0$ and $b>p\delta-n$ such that $q<\frac{(n+a)p}{s(n+b-\delta p)+(s-1)(p-1)}$. Then there exists a constants $\tau_0=\tau(\Omega)$ such that for each $\tau\in (0,\tau_0)$ and each $u\in C(\Omega)$, the weighted fractional Sobolev-Poincar\'e inequality \eqref{eq:weighted fractional Sob Poin} holds with $f=\rho^a$ and $g=\rho^b$.
\end{thm}

Note that the ranges for $b$ and $q$ in Theorem \ref{4.4t} are essentially sharp as Example 1.1 of \cite{G2017} indicates. The critial case $q=\frac{(n+a)p}{s(n+b-\delta p)+(s-1)(p-1)}$ can be achieved if $s=1$ or $p=1$ as in \cite[Remark 4.2]{G2017}, but it is technically very involved and we omit the details here. The case for other ranges of $s$ and $p$ remains open. We would like to comment that the validity of the fractional Sobolev-Poincar\'e inequality indeed gives control on the geometry of the domain $\Omega$; see \cite{BK1995} and \cite[Theorem 1.4]{G2017}.

Another class of irregular domains, which are largely considered in the literature (see for instance \cite{SS1990,G2017,JK2013,KOT2002}), are the so-called $\beta$-H\"older domains ($0<\beta\leq 1$), that is, a domain with a distinguished point $x_0\in\Omega$ such that for all $x\in \Omega$
\[k_{\Omega}(x,x_0)\leq\frac{1}{\beta}\log\frac{d(x_0,\partial\Omega)}{d(x,\partial\Omega)}+C,\]
where $k_\Omega$ is the quasihyperbolic distance in $\Omega$ (See Section 2 below for precise definition). The concept of H\"older domain was introduced by Smith and Stegenga \cite{SS1990} based on the earlier work of Becker and Pommerenke. Moreover, it is well-known that $\beta$-H\"older domains are $\frac{1}{\beta}$-John domains. The techniques for doing capacity estimates in H\"older domians are also well developed; see \cite{KOT2002,G2017}.

Our third main result can be regarded as a weighted fractional version of \cite[Theorem 1.4-1.5]{KOT2002} and \cite[Theorem 1.3]{G2017}.
\begin{thm}\label{t4.4}
Let $\Omega$ be a $\beta$-H\"older domain. Assume that $0<\delta<1$, $1\leq p\leq q$, $a\geq 0$ and $p\delta-n<b<(a+n)\beta\frac{p}{q}+p\delta-n$. Then
there exists a constants $\tau_0=\tau(\Omega)$ such that for each $\tau\in (0,\tau_0)$ and each $u\in C(\Omega)$, the weighted fractional Sobolev-Poincar\'e inequality \eqref{eq:weighted fractional Sob Poin} holds with $f=\rho^a$ and $g=\rho^b$.
\end{thm}
The ranges for $b$ and $q$ are essentially sharp as Example 1.2 of \cite{G2017} indicates.

Since our paper generalizes the corresponding results of \cite{HK1998} and \cite{G2017} in a nature way to the weighted fractional setting, many of the arguments used in this paper are similar to the ones in those papers. For the convenience of readers, we have included as many details as possible. The structure of this paper is as follows. In section 2, some basic lemmas and konwlegdes are introduced. When it comes to section 3, we prove the theorems connecting the weighted capacity-type inequalities and the weighted fractional Sobolev-Poincare inequalities. Then, the desired capacity-type inequalities for $\alpha$-John domains are proved in section 4. Moreover, we deduce the similar capacity-type inequalities as to $\beta$-Holder domains in section 5.
\bigskip
\section{Preliminary}

First of all, we fix some notation. Throughout this paper, $C(\cdot)$ denotes a constant, where ``$\cdot$" contains all parameters on which the constant depends. If there exists a constant $C\geq 1$ such that $A/C \le B \le CA$, then we write $A\asymp B$.

We shall always assume the dimension $n\geq 2$. The Euclidean distance between two points $x,y\in \rr^n$ is denoted by $|x-y|$, while $d(A,B)$ represents the Euclidean distance between two sets $A,B\subset\rr^n$. We also write $d(x,A)$ for the Euclidean distance between the point $x\in\rr^n$ and the set $A\subset\rr^n$. The notation $B\subset\subset\Omega$ simply means that $B$ is a subset of $\Omega$ with $d(B,\partial\Omega)>0$. The Euclidean diameter of a set $E\subset\rr^n$ is denoted by $d(E)$. For a measurable set $E\subset \rr^n$, $|E|$ represents the $n$-dimensional Lebesgue measure of $E$.

In $\rr^n$, the ball with the center $x$ and the radius $r$ is denoted by $B(x,r)$, while a cube $Q$ in $\rr^n$ is denoted as \[Q=\{x\in\rr^n:a_i<x_i<a_i+a,1\leq i\leq n,a>0\}.\]
The side-length of a cube $Q$ is denoted by $\ell(Q)$, that is, $\ell(Q)=a$ in the above definition. If $B=B(x,r)$ is a ball in $\rr^n$ and $t>0$ is a positive constant, then $tB$ denotes the ball $B(x,tr)$. For a cube $Q\in\rr^n$, $tQ$ is the cube with the same center, such that $\ell(tQ)=t\ell(Q)$.

The following fractional Sobolev-Poincar\'e inequality on balls or cubes in $\rr^n$ is well-known.

\begin{lem}[\cite{HSV2013}]\label{ddd}
Let $\Omega$ be a ball with radius $r$ or a cube with side-length $r$ in $\rr^n$. Assume that $1\leq p\leq q< \infty$, $\delta\in (0,1)$, $\tau\in(0,\infty)$. Then for all $u \in C(\Omega)$,
\begin{equation}
\int_{\Omega}{|u(x)-u_{\Omega}|}^qdx\leq C(n) r^{n+\frac{q}{p}(p\delta-n)}{\Big(\int_{\Omega}\int_{\Omega\cap B(x,\tau d(x,\partial \Omega))} \frac{{|u(x)-u(y)|}^p}{{|x-y|}^{n+p\delta}}dydx\Big)}^{\frac{q}{p}}.
\end{equation}
\end{lem}

The following chain lemma comes from \cite[the proof of Theorem 9]{HK1998}.

\begin{lem}\label{wwe}
Suppose that $\Omega\subset\rr^n$ is a $s$-John domain and $M>1$ is a fixed constant. Let $B_0=(x_0,\frac{\rho(x_0)}{4M})$, where $x_0$ is a point of $\Omega$. Then, there exists a constant $C>0$, depending only on $\Omega$, $M$ and $n$ such that, for any $x\in \Omega$, there exist finite balls $B_i=B(x_i,r_i)$($i=1,2,...,k$) with the following properties:
\begin{enumerate}
	\item $|B_i\cap B_{i-1}|\geq 1/C|B_i\cup B_{i-1}|$ for any $1\leq i\leq k$.
	
	\item $d(x,B_i)\leq C {r_i}^{\frac{1}{s}}$ as to any $0\leq i\leq k$.
	
	\item $d(B_i,\partial\Omega)\geq M r_i$ with respect to any $0\leq i\leq k$.
	
	\item $\sum_{i=0}^{k} \chi_{B_{i}} \leq C \chi_{\Omega}$.
	
	\item $|x-x_i|\leq C {r_i}^{\frac{1}{s}}$ for any $0\leq i\leq k-1$ and $B_k=B(x,\frac{\rho(x)}{4M})$.
	
	\item For any $r>0$, the number of balls $B_i$ with radius $r_i>r$ is less than $C r^{\frac{1-s}{s}}$ when $s>1$ and is less than $\log_2(Cr^{-1})$ when $s=1$.
	
	\item As to each $y\in B_i$, $r_i\asymp\rho(y)$, with respect to any $0\leq i \leq k$, where the constants only depend on $\Omega$, $M$ and $n$.
\end{enumerate}

\end{lem}

For the convenience of readers, we recall two covering lemmas that are needed in our later proofs. The first one is Vatali's covering lemma.
\begin{lem}\label{ppp}
Suppose $E$ is a bounded set in $\rr^n$. Let $E\subset\cup_{j\in J} B_j$, where $\{B_j\}_{j\in J}$ is a family of balls. Then there exists a subfamily of countable pairwise disjoint balls (possibly finite) $B_1$,$B_2$,... such that $E\subset \cup^{\infty}_{i=1} 5B_i$.
\end{lem}

The second one is called the Besicovitch covering lemma.

\begin{lem}\label{sss}
Assume that $E\subset\rr^n$ is a bounded set. Let $B_x=B(x,r_x)$ for $x\in E$. Then we may find a sequence of points (possibly finite) $x_i\in E$ ($i=1,2,...$) such that $E\subset \bigcup^{\infty}_{i=1} B_{x_i}$ with the property that no point of $\rr^n$ belongs to more than $C(n)$ such balls, where $C(n)$ is a constant only depending on the dimension $n$.
\end{lem}

Next, we introduce the quasihyperbolic distance. Let $\Omega\subset\rr^n$ be a proper domain, where a proper domain means a domain $\Omega\subsetneqq\rr^n$. The quasihyperbolic distance between two points $x,y\in\Omega$ is defined to be
\[k_\Omega(x,y)=\inf_\gamma\int_\gamma\frac{ds}{\rho(x)},\]
where the infimum is taken over all curves $\gamma$ in $\Omega$ connecting $x$ and $y$.

Recall that the Whitney decomposition $\mathcal{W}(\Omega)$ of a domain $\Omega$ is a collection of cubes in $\Omega$ such that they are pairwise disjoint, with the property that the union of the closure of these cubes are the whole $\Omega$, satisfying that \[d(Q)\leq d(Q,\partial\Omega)\leq 4d(Q),\]
for all $Q\in \mathcal{W}(\Omega)$.

In our article, we need the Whitney-type decomposition of a proper domain $\Omega\subset\rr^n$ given as follows.
\begin{lem}[\cite{S1970}]\label{lemma:Whitney}
Let $\Omega$ be a proper domain of $\rr^n$. Suppose $k\geq 2$ is an integer. Thus, we have a Whitney-type decomposition $\mathcal{V}_k(\Omega)=\{Q_i:i\geq 1\}$ of some cubes, satisfying the following three properties:
\begin{enumerate}
	\item $\bigcup_i \overline{Q_i}=\Omega$;
	
	\item $Q_i\cap Q_j=\emptyset$ for any $i\neq j$;
	
	\item $kd(Q_i)\leq d(Q_i,\partial\Omega)\leq 5kd(Q_i)$ for all $i\geq 1$.
\end{enumerate}
\end{lem}

We fix a Whitney-type decomposition of $\Omega$ and a cube $Q_0$ in that decomposition and denote the center of $Q_0$ by $x_0$. For each Whitney-type cube $Q$, we choose a quasihyperbolic geodesic joining $x_0$ to the center of the cube $Q$, and the set of the Whitney-type cubes that have non-empty intersection with this geodesic is denoted by $P(Q)$. Then, the shadow of the Whitney-type cube $Q$ is defined as \[S(Q)=\bigcup_{\{Q_1\in \mathcal{W}(\Omega):Q\in P(Q_1)\}}Q_1.\]
Then, a lemma controls the size of $P(Q)$ with respect to the $\beta$-Holder domains ($0<\beta\leq 1$) was proved in \cite{KOT2002}.

\begin{lem}[\cite{KOT2002}]\label{2.7l}
For a $\beta$-H\"older domain $\Omega\subset\rr^n$ ($0<\beta\leq 1,n\geq 2$) and a constant  $\varepsilon>0$, there is a constant $C=C(\varepsilon,d(\Omega),n,k)$, such that
\[\sum_{Q\in P(Q_1)}{|Q|}^\varepsilon\leq C,\]
for any $Q_1\in \mathcal{V}_k(\Omega)$.
\end{lem}

%

\bigskip
\section{Proofs of Theorem \ref{thm:capactity characterization}}

In this section, we shall prove Theorem \ref{thm:capactity characterization}. We first prove a weaker version of \eqref{eq:weighted fractional Sob Poin} and the proof here is similar to \cite[Proof of Theorem 1.1]{G2017}.

\begin{thm}\label{4.1l}
	Let $f,g$ be two positive continuous functions on an open set $\Omega\subset \R^n$ with $\int_{\Omega}f(x)dx<\infty$. Let $1\leq p \leq q<\infty$, $\delta\in (0,1)$ and $\tau\in(0,\infty)$. Assume that for any ball $B\subset\subset\Omega$, there exists a constant  $C=C(\Omega,B,f,g,p,q,\delta,\tau)$ such that
	\[
	\Big(\int_Af(x)dx\Big)^{\frac{p}{q}}\leq C \inf_{u}\int_\Omega g_u^\tau(x)g(x)dx,
	\]
	for any admissible set $A\subset\Omega$ with $A\cap B=\emptyset$, where the infimum is taken with respect to all the functions $u\in C(\Omega)$, satisfying that $u|_A\geq 1$ and $u|_B=0$. Then for any ball $B\subset\subset\Omega$, there is a constant $C$, such that for any $u\in C(\Omega)$ with $u|_B=0$,
	\[
	\Big(\int_\Omega{|u(x)|}^qf(x)dx\Big)^{\frac{1}{q}}\leq C\Big(\int_\Omega g_u^\tau(x)g(x)dx\Big)^{\frac{1}{p}}.
	\]
\end{thm}

\begin{proof}
Fix an arbitrary ball $B_0\subset\subset\Omega$. Assume that $u\geq 0$ is a continuous function with $u|_{B_0}=0$. For any $j\in\zz$, define \[u_j(x)=\min\{2^j,\max\{0,u(x)-2^j\}\}.\]
It is clear that $2^{-j}u_j|_{B_0}=0$ and $2^{-j}u_j|_{F_j}\geq 1$, where $F_j=\{x\in\Omega:u(x)\geq 2^{j+1}\}.$ Moreover, for $u_j$, we can use the assumption to find:
\[{\Big(\int_{F_j} f(x)dx\Big)}^{\frac{p}{q}}\leq C \int_{\Omega} g_{2^{-j}u_j}gdx,\]
where for simplicity we omitted the superscript $\tau$ in $g_u^\tau$. For $k\in\zz$, we define $A_k:=F_{k-1}\backslash F_k$. For $y\in\Omega$, set $B_y=B(y,\tau \rho(y))$. \begin{equation}\label{4321}
	\begin{aligned}\int_{\Omega}{|u(x)|}^qf(x)dx&\leq \sum_{k=-\infty}^{+\infty} 2^{(k+1)q}\int_{A_k}f(x)dx \\& \leq C \sum_{k=-\infty}^{+\infty} {\Big(\int_{\Omega} g_{u_k}gdx\Big)}^{\frac{q}{p}}\\& \leq C {\Big(\sum_{k=-\infty}^{+\infty}\int_{\Omega} g_{u_k}gdx\Big)}^{\frac{q}{p}}\\& \leq C {\Big(\sum_{k=-\infty}^{+\infty}(I_1^k+I_2^k)\Big)}^{\frac{q}{p}}\end{aligned},
\end{equation}
where \[I_1^k=\sum_{i\leq k}\sum_{j\geq k}\int_{A_i}\int_{A_j\cap B_y}\frac{{|u_k(y)-u_k(z)|}^pg(y)}{{|y-z|}^{n+p\delta}}dzdy,\]
and
\[I_2^k=\sum_{i\geq k}\sum_{j\leq k}\int_{A_i}\int_{A_j\cap B_y}\frac{{|u_k(y)-u_k(z)|}^pg(y)}{{|y-z|}^{n+p\delta}}dzdy.\]

For any $y\in A_i$ and $z\in A_j$ with $j-1>i$, \[|u(y)-u(z)|\geq |u(y)|-|u(z)|\geq 2^j-2^{i+1}\geq 2^{j-1}.\]
Moreover, for such $y,z$, \begin{equation}\label{12345}
|u_k(y)-u_k(z)|\leq 2^{k+1} \leq 4\:2^{k-j}|u(y)-u(z)|.
\end{equation}
Since, for each $k\in\zz$, \[|u_k(z)-u_k(y)|\leq |u(y)-u(z)|,\] then for any $i\leq k\leq j$, $y\in A_i$ and $z\in A_j$,
\eqref{12345} holds.
Thus, \[\sum_{k=-\infty}^{+\infty}I_1^k\leq 4^p\sum_{k=-\infty}^{+\infty}\sum_{i\leq k}\sum_{j\geq k} 2^{p(k-j)}\int_{A_i}\int_{A_j\cap B_y}\frac{g(y){|u(y)-u(z)|}^p}{{|y-z|}^{n+p\delta}}dzdy.\]
As $\sum_{k=i}^j2^{p(k-j)}\leq {(1-2^{-p})}^{-1},$ by changing the order of summation, we obtain
\[\sum_{k=-\infty}^{+\infty}I_1^k\leq \frac{4^p}{1-2^{-p}}\int_{\Omega}g_u(y)g(y)dy.\]
By a similar reason, we obtain that \[\sum_{k=-\infty}^{+\infty}I_2^k\leq C\int_{\Omega}g_u(y)g(y)dy.\]
Therefore, by \eqref{4321} and the above two equations, we conclude that
\[\int_{\Omega}{|u(x)|}^qf(x)dx\leq C{\Big(\int_{\Omega}g_u(y)g(y)dy\Big)}^{\frac{q}{p}}.\]
From the above conclusion, it is easy to prove  that for any ball $B\subset\subset\Omega$ there exists a constant $C$ such that for each $u\in C(\Omega)$ with $u|_B=0$, \[{\Big(\int_{\Omega}{|u|}^qfdx\Big)}^{\frac{1}{q}}\leq C {\Big(\int_{\Omega}g_ugdx\Big)}^{\frac{1}{p}}.\]

\end{proof}

We next show that the conclusion of Theorem \ref{4.1l} implies the weighted fractional Sobolev-Poincar\'e inequality \eqref{eq:weighted fractional Sob Poin}.
\begin{thm}\label{4.3l}
	Let $f,g$ be two positive continuous functions on an open set $\Omega\subset \R^n$ with $\int_{\Omega}f(x)dx<\infty$. Let $1\leq p \leq q<\infty$, $\delta\in (0,1)$. Then there exists a constant $\tau_0=\tau(\Omega)$ such that if for any ball $B\subset\subset\Omega$, there is a constant $C$ so that for any $u\in C(\Omega)$ with $u|_B=0$, it holds
	\[
	\Big(\int_\Omega{|u(x)|}^qf(x)dx\Big)^{\frac{1}{q}}\leq C{\Big(\int_\Omega g_u^\tau(x)g(x)dx\Big)}^{\frac{1}{p}},
	\]
	then there exists a constant $C$ such that for each $\tau\in (0,\tau_0)$ and each $u\in C(\Omega)$,
	\[{\Big(\int_\Omega{|u(x)-u_{\Omega,f}|}^{q}f(x)dx\Big)}^{\frac{1}{q}}\leq C {\Big(\int_\Omega g_u^\tau(x)g(x)dx\Big)}^{\frac{1}{p}}.\]
	
\end{thm}

\begin{proof}
Since $\Omega$ is bounded, there exist a constant $k=k(\Omega)$ and a ball $B_0=B(x_0,r_0)$ such that $3B_0\subset\subset\Omega\subset kB_0.$ Set $\tau_0=\tau_0(\Omega)=\frac{1}{k-3}$ and assume $\tau<\tau_0$.

Take a smooth function $\phi$ such that $0\leq\phi\leq 1$ on $\Omega$, $\phi|_{B_0}=1$ and the support of $\phi$ is contained in $2B_0$. For any $v\in C(\Omega)$, we may decompose $v-v_{3B_0}$ as \[v-v_{3B_0}=\phi(v-v_{3B_0})+(1-\phi)(v-v_{3B_0})=:v_1+v_2.\]
It is clear that $v_2|_{B_0}=0$ and $v_2|_{\Omega\backslash 2B_0}=v-v_{3B_0}.$
Therefore, \[\int_\Omega {|v-v_{3B_0}|}^qf\leq C\int_{3B_0}{|v_1|}^qf+C\int_\Omega {|v_2|}^qf=I_1+I_2.\]
By the fractional Sobolev-Poincar\'e inequality on balls (see Lemma \ref{ddd}), we have
\[I_1\leq C\int_{3B_0} {|v-v_{3B_0}|}^qf\leq C \int_{3B_0} {|v-v_{3B_0}|}^q\leq C {\Big(\int_{3B_0} g_v\Big)}^{\frac{q}{p}}\leq C {\Big(\int_{3B_0} g_vg\Big)}^{\frac{q}{p}}.\]
Furthermore, since $v_2|_{B_0}=0$, by the assumption of the theorem,  \[\begin{aligned}
	I_2^{\frac{p}{q}}&\leq  C\int_{\Omega}g_{v_2}g= C\int_{\Omega\backslash 3B_0}g_{v_2}g+C\int_{3B_0}g_{v_2}g \\
	&=C\int_{\Omega\backslash 3B_0}g_vg+C\int_{3B_0}g_{v_2}g,
\end{aligned}\]
where in  the last line of the above inequality we used the fact that $g_{v_2}=g_{v}$ on $\Omega\backslash 3B_0$. This follows from the fact that $v=v_2$ on $\Omega\backslash 2B_0$ and that if $x\in \Omega\backslash 3B_0$ and $y\in B(x,\tau\rho(x))$, then $y\notin 2B_0$, since otherwise we would have \[|x-y|<\tau\rho(x)\leq \tau(k-3)r_0\leq |x-y|,\]
which is clearly a contradiction.

We next estimate the term $\int_{3B_0}g_{v_2}g.$
Note that \[
\begin{aligned}
|v_2(x)-v_2(y)|\leq |(1-\phi(x))(v(x)&-v_{3B_0})-(1-\phi(y))(v(x)-v_{3B_0})|\\
&+|(1-\phi(y))(v(x)-v_{3B_0})-(1-\phi(y))(v(y)-v_{3B_0})|.
\end{aligned}
\]
An easy computation leads to
\[g_{v_2}(x)=\int_{\Omega\cap B(x,\tau\rho(x))}\frac{{|v_2(x)-v_2(y)|}^p}{{|x-y|}^{n+\delta p}}dy\leq C{|v(x)-v_{3B_0}|}^p+Cg_v(x),\]
where $C$ depends only on the data of $\phi$.
Therefore, applying the fractional Sobolev-Poincar\'e inequality on balls again, we infer that
\[\int_{3B_0}g_{v_2}g\leq C\int_{3B_0}g_{v}g+C\int_{3B_0}{|v-v_{3B_0}|}^pg\leq C\int_{3B_0}g_{v}g.\]
The conclusion follows from the previous estimates by noticing the elementary inequality
$$\int_{\Omega}|v-v_{\Omega,f}|^qfdx\leq C(q)\int_{\Omega}|v-v_{3B_0}|^qfdx.$$

\end{proof}

\begin{proof}[Proof of Theorem \ref{thm:capactity characterization} ]
	This follows immediately from Theorem \ref{4.1l} and \ref{4.3l}.
\end{proof}
\bigskip
\section{Weighted fractional Sobolev-Poincar\'e inequality in $s$-John domains}
We first prove a weighted capacity estimate for $s$-John domains.
\begin{thm}\label{1.1t}
	Suppose that $\Omega$ is an $s$-John domain. Let $0<\delta<1$, $\tau>0$, $1\leq p\leq q$, $a\geq 0$, $b>p\delta-n$ be constants, satisfying that $q<\frac{(n+a)p}{s(n+b-\delta p)+(s-1)(p-1)}$. Then for each ball $B\subset\subset\Omega$ there exists a constant $C=C(\Omega,p,q,a,b,\delta,\tau,B)$ such that
	\begin{equation}
		{\Big(\int_A {\rho(x)}^adx\Big)}^{\frac{p}{q}}\leq C\:\inf\int_{\Omega}g_u^\tau(x){\rho(x)}^b dx,
	\end{equation}
	for every admissible subset $A\subset\Omega$, satisfying that $A\cap B=\emptyset$. Here, the infimum is taken over all the functions $u\in C(\Omega)$ such that $u|_A\geq 1$ and $u|_B=0$.
\end{thm}

\begin{proof}
The proof here is simialr to \cite[Proof of Theorem 1.2]{G2017}. As \[\frac{n+a}{s(n+b-\delta p)+(s-1)(p-1)}> \frac{q}{p}\geq 1,\] we may choose $\Delta>0$ so that \[\frac{n+a}{s(n+b-\delta p)+(s-1)(p-1)+2\Delta}=\frac{q}{p}.\]

When dealing with $B_0\subset\subset\Omega$, without loss of generality, we may assume that $B_0=B(x_0,\frac{\rho(x_0)}{4M})$, where $M>1$ and $M>\frac{2}{\tau}$. In order to prove this theorem, it suffices to verify that there exists a constant $C$ such that for every admissible subset $A$ of $\Omega$ with $A\cap B_0=\emptyset$, it holds
\begin{equation*}
		{\Big(\int_A {\rho(x)}^adx\Big)}^{\frac{p}{q}}\leq C \int_{\Omega}g_u(x){\rho(x)}^b dx,
	\end{equation*}
	where $u\in C(\Omega)$, $u|_A\geq1$ and $u|_{B_0}=0$. Up to a similarity of $\rr^n$, we may assume that $d(\Omega)=1$.
	For any $x\in A$, there exists a chain of finite balls $B_0$,$B_1$,...,$B_k$ satisfying the conditions of Lemma~\ref{wwe}. In particular, for any $y\in B_i$, \[B_i\subset B(y,\tau \rho(y)).\]
	To see this, fix the point $y$ and take an arbitrary point  $z\in B_i$. Then, by Lemma~\ref{wwe}, \[|y-z|\leq |y-x_i|+|x_i-z|\leq 2r_i\leq\frac{2}{M}d(B_i,\partial\Omega)\leq\frac{2}{M}\rho (y)<\tau\rho (y).\]
	
	For $x\in\Omega$, denote by $B_x$ the ball $B(x,\frac{\rho(x)}{4M})$ and set
	\[\mathscr C=\{x\in A: u_{B_x}\geq\frac{1}{2}\},\ \mathscr D=\{x\in A: u_{B_x}<\frac{1}{2}\}.\]
	Then, we have
	\begin{equation}\label{aaa}
		\int_A {\rho(x)}^adx=\int_\mathscr C{\rho(x)}^a+\int_\mathscr D{\rho(x)}^adx.
	\end{equation}
	
	For any $x\in \mathscr C$, we have taken a chain of balls $B_0$,$B_1$,...,$B_k$. According to Lemma~\ref{wwe},
	\begin{equation}\label{poi}
		\frac{1}{2}\leq|u_{B_k}-u_{B_0}|\leq\sum_{i=0}^{k-1}(|u_{B_i}-u_{B_i\cap B_{i+1}}|+|u_{B_{i+1}}-u_{B_i\cap B_{i+1}}|)\leq C \sum_{i=0}^{k}\frac{1}{|B_i|}\int_{B_i}|u-u_{B_i}|dy.
	\end{equation}
	Then, we fix a ball $B_i$ and apply H\"older's inequality, \[\frac{1}{|B_i|} \int_{B_i} |u(y)-u_{B_i}|dy\leq\frac{1}{|B_i|} \int_{B_i} {\Big(\frac{1}{|B_i|} \int_{B_i}{|u(y)-u(z)|}^{p}dz\Big)}^{\frac{1}{p}}dy.\]
	As $|B_i|\geq C {|y-z|}^n$ for all $y,z\in B_i$, \[\frac{1}{|B_i|} \int_{B_i} |u(y)-u_{B_i}|dy\leq C {|B_i|}^{\frac{\delta}{n}-1} \int_{B_i} {\Big( \int_{B_i}\frac{{|u(y)-u(z)|}^{p}}{{|y-z|}^{n+p\delta}}dz\Big)}^{\frac{1}{p}}dy.\]
	We may add the above euqations to obtain
	\[\sum_{i=0}^{k}\frac{1}{|B_i|}\int_{B_i}|u-u_{B_i}|dy\leq C\sum_{i=0}^{k}{|B_i|}^{\frac{\delta}{n}-1}\int_{B_i}{{\Big({\int_{B_i} \frac{{|u(y)-u(z)|}^p}{{|y-z|}^{n+p\delta}}dz}\Big)}^{\frac{1}{p}}dy}.\]
	Since $B_i\subset B(y,\tau\rho(y))\cap\Omega$,
	\[\begin{aligned}\sum_{i=0}^{k}\frac{1}{|B_i|}\int_{B_i}|u-u_{B_i}|dy&\leq C\sum_{i=0}^{k}{|B_i|}^{\frac{\delta}{n}-1}\int_{B_i}{{\Big({\int_{B(y,\tau \rho(y))\cap \Omega} \frac{{|u(y)-u(z)|}^p}{{|y-z|}^{n+p\delta}}dz}\Big)}^{\frac{1}{p}}dy}\\&\leq C \sum_{i=0}^k r_i^{\delta-\frac{n}{p}}{\Big(\int_{B_i} g_u(y)dy\Big)}^\frac{1}{p}.\end{aligned}\]
	 Combining this with \eqref{poi}, we arrive at
	\begin{equation}\label{eq1}
		C\leq \sum_{i=0}^k r_i^{\delta-\frac{n}{p}}{\Big(\int_{B_i} g_u(y)dy\Big)}^\frac{1}{p}.
	\end{equation}
	Write $\kappa =\frac{(s-1)(p-1)+\Delta}{sp}$. Applying Holder's inequality, we conclude
	\[C\leq {\Big(\sum_{i=0}^k r_i^{\frac{\kappa p}{p-1}}\Big)}^{\frac{p-1}{p}}{\Big(\sum_{i=0}^k r_i^{p(-\kappa +\delta-\frac{n}{p})}\int_{B_i} g_u(y)dy\Big)}^\frac{1}{p}.\] Lemma~\ref{wwe} implies that when $s>1$
	\[\sum_{i=0}^k r_i^{\frac{\kappa p}{p-1}}\leq\sum_{i=0}^{\infty}{(2^{-i})}^{\frac{\kappa p}{p-1}}2^{\frac{i(s-1)}{s}}<C,\]
    and we could obtain the same conclusion for the easier case $s=1$.
	Combining the above two inequalities, we infer that
	\begin{equation}\label{pop}
		\sum_{i=0}^k r_i^{p(-\kappa +\delta-\frac{n}{p})}\int_{B_i} g_u(y)dy\geq C.
	\end{equation}
	By Lemma~\ref{wwe}, for any $0\leq i\leq k-1$ and $y\in B_i$, $C r_i\geq {|x-y|}^s$. Note that $-\kappa p+\delta p-n-b<0$. This implies that for any $y\in B_i$,
	\[{r_i}^{-\kappa p+\delta p-n-b}\leq C{|x-y|}^{s(-\kappa p+\delta p-n-b)}.\]
	Applying Lemma~\ref{wwe},
	\[{r_i}^{-\kappa p+\delta p-n}\leq C{\rho(y)}^b{|x-y|}^{s(-\kappa p+\delta p-n-b)}.\]
	As for $y\in B_i\cap (2^{j+1}B_k\backslash 2^jB_k)$, \[|x-y|\asymp 2^jr_k,\]where $0\leq j\leq |{\log}_2{r_k}|$. Since $d(\Omega)=1$, we only need to consider such $j$. For such $y$, \[{r_i}^{-\kappa p+\delta p-n}\leq C{\rho(y)}^b{(2^jr_k)}^{s(-\kappa p+\delta p-n-b)}.\]
	Since $d(\Omega)=1$, ${r_k}^{-\kappa p+\delta p-n-b}\geq 1$. Moreover, as Lemma~\ref{wwe} implies, for $y\in B_k$, \[{r_k}^{-\kappa p+\delta p-n}\asymp {r_k}^{-\kappa p+\delta p-n-b}{\rho(y)}^b\leq {r_k}^{s(-\kappa p+\delta p-n-b)}{\rho(y)}^b.\]
	Thus, the above two equations, the conditions within Lemma~\ref{wwe} and \eqref{pop} give us that \[\begin{aligned}C& \leq\sum_{i=0}^k {r_i}^{p(-\kappa+\delta-\frac{n}{p})} \int_{B_i}g_u(y)dy \\& \leq C {r_k}^{s(-\kappa p+\delta p-n-b)}\int_{B_k}g_u(y){\rho(y)}^bdy+C \sum_{j=0}^{|\log_2r_k|}{(2^jr_k)}^{s(-\kappa p+\delta p-n-b)}\int_{\Omega\cap (2^{j+1}B_k\backslash 2^jB_k)}g_u(y){\rho(y)}^{b}dy\\& \leq C\sum_{l=0}^{|\log_2r_k|+1}{(2^lr_k)}^{s(-\kappa p+\delta p-n-b)}\int_{2^lB_k\cap \Omega}g_u(y){\rho(y)}^{b}dy. \end{aligned} \]
	Since $\Delta>0$, we have \[\sum_{l=0}^{|\log_2r_k|+1}{(2^lr_k)}^{\Delta}\leq {r_k}^{\Delta}\sum_{l=-\infty}^{|\log_2 r_k|+1}2^{l\Delta}\leq {r_k}^{\Delta}2^{(|\log_2r_k|+1)\Delta}<C.\]
	Comparing the above two inequalities, there exists an $l\in[0,|\log_2r_k|+1]$ such that \begin{equation}\label{eq3t}{(2^lr_k)}^{\Delta}<C{(2^lr_k)}^{s(-\kappa p+\delta p-n-b)}\int_{2^lB_k\cap \Omega}g_u(y){\rho(y)}^{b}dy.\end{equation}
	In other words, there is an $R_x\geq r_k\geq C\rho(x)$, satisfying \[C{R_x}^{s(\kappa p-\delta p+n+b)+\Delta}\leq\int_{B(x,R_x)\cap \Omega}g_u(y){\rho(y)}^{b}dy.\]
	Using Lemma~\ref{ppp} for the family $\{B(x,R_x)\}_{x\in \mathscr C}$ covering $\mathscr C$, we obtain a sequence of disjoint balls $B_{(1)}$, $B_{(2)}$,$B_{(3)}$,... so that $\mathscr C\subset \cup_{i=1}^{\infty}5B_{(i)}$. Denote the radius of $B_{(i)}$ by $r_{(i)}$. For $y\in B(x,R_x)\cap\Omega$, it is easy to observe that $\rho(y)\leq C R_x$, from which it follows ${\rho(y)}^a\leq C{R_x}^a$, where we used the condition $a\geq 0$. Therefore, we may estimate as follows:
	\begin{equation}\label{qqq}
		\begin{aligned}\int_{\mathscr C}{\rho(x)}^adx&\leq \sum_{i=1}^{\infty}\int_{5B_{(i)}\cap\Omega} {\rho(x)}^adx\\&\leq C\sum_{i=1}^{\infty} r_{(i)}^{n+a}\\&\leq C \sum_{i=1}^{\infty}{(\int_{B_{(i)}\cap \Omega}g_u(y){\rho(y)}^{b}dy)}^{\frac{n+a}{s(n+b+\kappa p-\delta p)+\Delta}}\\&\leq C {(\sum_{i=1}^{\infty}\int_{B_{(i)}\cap \Omega}g_u(y){\rho(y)}^{b}dy)}^{\frac{n+a}{s(n+b+\kappa p-\delta p)+\Delta}}.\end{aligned}\end{equation}
	In the last inequality, we have applied the fact that \[\frac{n+a}{s(n+b+\kappa p-\delta p)+\Delta}=\frac{n+a}{s(n+b-\delta p)+(s-1)(p-1)+2\Delta}\geq 1.\]
Therefore, we get the desired estimate of the first term of \eqref{aaa}, that is, \begin{equation}\label{sas}
		\int_{\mathscr C}{\rho(x)}^adx\leq C {\Big(\int_{\Omega}g_u(y){\rho(y)}^{b}dy\Big)}^{\frac{q}{p}},
	\end{equation}
	which follows from \eqref{qqq} and the above equation.
	
	Next, we are going to estimate the second part of \eqref{aaa}. Remember that, for $x\in\Omega$,  $B_x=B(x,\frac{\rho(x)}{4M}).$ Then $\mathscr D\subset\cup_{x\in \mathscr D}B_x.$ Moreover, by Lemma~\ref{sss}, there exists a sequence of balls $B_{x_1}$,$B_{x_2}$,$B_{x_3}$,... such that $\mathscr D\subset \cup_{i=1}^{\infty} B_{x_i}$. Since $u_{B_{x_i}}\leq \frac{1}{2}$ and $u|_A\geq 1$, for $y\in B_{x_i}\cap A$, \[{|u(y)-u_{B_{x_i}}|}^q\geq\frac{1}{2^q},\]
	for all $i\geq 1$.
	By Lemma~\ref{ddd} and the above relationship, we obtain that \[|A\cap B_{x_i}|\leq C\int_{A\cap B_{x_i}}{|u-u_{B_{x_i}}|}^q\leq Cr_{x_i}^{n+\frac{q}{p}(p\delta-n)}{\Big(\int_{B_{x_i}}g_u(y)dy\Big)}^{\frac{q}{p}}.\]
	Thus, we have
	\begin{equation}\label{xxx}\begin{aligned}\int_\mathscr D{\rho(y)}^ady&\leq \sum_{i=1}^{\infty}\int_{B_{x_i}\cap A}{\rho(y)}^ady\\&\leq C\sum_{i=1}^{\infty}r_{x_i}^a|A\cap B_{x_i}|\\&\leq C\sum_{i=1}^{\infty}r_{x_i}^{a+n+\frac{q}{p}(p\delta-n)}{\Big(\int_{B_{x_i}}g_u(y)dy\Big)}^{\frac{q}{p}}
			\\&\leq C {\Big(\int_{\Omega}{\rho(y)}^{[a+n+\frac{q}{p}(p\delta-n)]\frac{p}{q}}g_u(y)dy\Big)}^{\frac{q}{p}}. \end{aligned}
	\end{equation}
	Then, by the condition of the main theorem, it easily follows that \[n+a\geq \frac{q}{p}(b+n-p\delta),\]
	from which we obtain that, \[[a+n+\frac{q}{p}(p\delta-n)]\frac{p}{q}\geq b.\]
	Since $d(\Omega)=1$, we obtain by \eqref{xxx} that \[\int_\mathscr D{\rho(y)}^ady\leq C {\Big(\int_{\Omega}{\rho(y)}^{b}g_u(y)dy\Big)}^{\frac{q}{p}}. \]
	Combining the above equation with \eqref{sas}, we complete the whole proof.
\end{proof}

A straightforward consequence of Theorme \ref{1.1t} is the following weighted fractional Sobolev-Poincar\'e  inequality in $s$-John domains.
\begin{thm}\label{4.2t}
Let $\Omega$ be an $s$-John domain. Assume that $0<\delta<1$, $\tau>0$, $1\leq p\leq q$, $a\geq 0$ and $b>p\delta-n$, satisfying that $q<\frac{(n+a)p}{s(n+b-\delta p)+(s-1)(p-1)}$. Then for any ball $B\subset\subset\Omega$, there is a constant $C$, such that, for any $u\in C(\Omega)$ with $u|_B=0$, \[{\Big(\int_\Omega{|u(x)|}^q{\rho(x)}^adx\Big)}^{\frac{1}{q}}\leq C{\Big(\int_\Omega g_u^\tau(x){\rho(x)}^bdx\Big)}^{\frac{1}{p}}.\]
\end{thm}
\begin{proof}[Proof of Theorem~\ref{4.2t}]
This follows immediately from Theorem~\ref{1.1t} and Theorem~\ref{4.1l}.
\end{proof}

\begin{proof}[Proof of Theorem~\ref{4.4t}]
This is a direct consequence of Theorem~\ref{1.1t} and Theorem~\ref{thm:capactity characterization}.
\end{proof}

\section{Weighted fractional Sobolev-Poincar\'e inequality in $\beta$-H\"older domains}

As in the previous section, we first prove a weighted capacity estimate in $\beta$-H\"older domains.
\begin{thm}\label{1.2t}
	Let $\Omega$ be a $\beta$-H\"older domain ($0<\beta\leq 1$). Assume that $\tau>0$, $0<\delta<1$, $1\leq p\leq q$, $a\geq 0$ and $p\delta-n<b<(a+n)\beta\frac{p}{q}+p\delta-n.$
	Then for any ball $B\subset\subset\Omega$, there is a constant $C(\Omega,a,b,p,q,\delta,\tau,B)$ such that \[{\Big(\int_A\rho(x)^adx\Big)}^{\frac{p}{q}}\leq C\inf \int_\Omega g_u^\tau(x)\rho(x)^bdx,\]
	for any admissible subset $A\subset\Omega$ satisfying that $A\cap B=\emptyset$, where the infimum is over all the functions $u\in C(\Omega)$ with $u|_A\geq 1$ and $u|_B=0$.
\end{thm}

\begin{proof}[Proof of Theorem~\ref{1.2t}]
Our assumption implies that, there exists a constant $\varepsilon>0$, such that \[(a+n)\frac{p}{q}\beta>b-\delta p+n+\varepsilon(p-1)>0.\]
Moreover, there is a $t>0$ such that \[\frac{a+n}{t+\frac{b-p\delta+n+\varepsilon(p-1)}{\beta}}=\frac{q}{p}.\]
Fix an arbitrary ball $B\subset\subset\Omega$. For any positive integer $m$, applying Lemma~\ref{lemma:Whitney} for $\Omega$, we obtain that there exists a Whitney-type decomposition $\mathcal{V}_m(\Omega)$. Taking $m$ large enough, we may assume that there is a cube $Q_0$ of a Whitney-type decomposition $\mathcal{V}_m(\Omega)$ such that $Q_0\subset B$. By taking $m$ even larger, we suppose that $m>\frac{6}{5\tau}+\frac{1}{10}$. Furthermore, let the center of $Q_0$ be $x_0$ and, without loss of generality, $d(\Omega)=1$. It suffice to prove that there is a constant $C$ such that, for any admissible set $A\subset\Omega$ with $A\cap Q_0=\emptyset$, it holds that
\[{\Big(\int_A{\rho(x)}^adx\Big)}^{\frac{p}{q}}\leq C\inf\int_\Omega g_u^\tau(x){\rho(x)}^bdx,\]
where the infimum is taken with respect to all the functions $u\in C(\Omega)$ with $u|_A\geq 1$ and $u|_{Q_0}=0$. For $x\in\Omega \backslash(\cup_{Q\in\mathcal{V}_k}\partial Q)$, define $Q(x)$ as the unique Whitney cube such that $x\in Q(x)$. Since $|\cup_{Q\in\mathcal{V}_k(\Omega)}\partial Q|=0$, we would always ignore this set in the following.  Let \[\mathcal{C}=\{x\in A:u_{Q(x)}\geq \frac{1}{2}\},\] \[\mathcal{D}=\{x\in A:u_{Q(x)}>\frac{1}{2}\}.\]
For any $x\in\mathcal{C}$, choose a quasihyperbolic geodesic $\gamma$, joining $x_0$ and $x$. Denote the Whitney-type cubes meeting $\gamma$ by $Q_0$,$Q_1$,...,$Q_k$ with center $x_0$,$x_1$,...,$x_k$ such that $x\in Q_k$ and
\[\overline{Q_i}\cap\overline{ Q_{i+1}}\neq\emptyset,\]
for $0\leq i\leq k-1$. Then, define $\frac{6}{5}Q_1$,$\frac{6}{5}Q_2$,...,$\frac{6}{5}Q_{k-1}$ as $Q_1^{\prime}$,$Q_2^{\prime}$,...,$Q_{k-1}^{\prime}$ with $Q_0^{\prime}=Q_0$ and $Q_k^{\prime}=Q_k$. Thus, for any integer $0\leq i\leq k$, $y\in Q_i^{\prime}$ and $z\in Q_i^{\prime}$, it holds that \[|y-z|\leq |y-x_i|+|z-x_i|\leq d(Q_i^{\prime})\leq \frac{6}{5}d(Q_i),\]
and that \[\rho(y)\geq d(Q_i^{\prime},\partial\Omega)\geq d(Q_i,\partial\Omega)-\frac{1}{10}d(Q_i)\geq (m-\frac{1}{10})d(Q_i).\]
Thus, \[|y-z|\leq \frac{\frac{6}{5}}{m-\frac{1}{10}}\rho(y)<\tau \rho(y).\]
In other words, \[Q_i^{\prime}\subset B(y,\tau\rho(y)),\]
for any $y\in Q_i^{\prime}.$
Furthermore, it is easy to prove that if $Q_1$ and $Q_2$ are two Whitney-type cubes such that $\overline{Q_1}$ and $\overline{Q_2}$ have nonempty intersection, then their Euclidean diameters are comparable. Therefore, there is a constant $C$ such that \[|Q_i^{\prime}\cap Q_{i-1}^{\prime}|\geq \frac{1}{C}|Q_i^{\prime}\cup Q_{i-1}^{\prime}|\] for any $1\leq i\leq k.$
Thus, by the similar method as \eqref{eq1}, we have the following equation:
\[C\leq \sum_{i=0}^k r_i^{\delta-\frac{n}{p}}{\Big(\int_{Q_i^{\prime}} g_u(y)dy\Big)}^\frac{1}{p},\]
where $r_i$ denotes the length of the cube $Q_i^{\prime}$ as to any $0\leq i\leq k$.
Then, by Lemma~\ref{2.7l} and H\"older's inequality, we obtain that for $\varepsilon>0$, \begin{equation}\label{eq2t}
\begin{aligned}C&\leq {\Big(\sum_{i=0}^{k}r_i^{\varepsilon}\Big)}^{\frac{p-1}{p}}{\Big(\sum_{i=0}^{k}r_i^{\Delta}\int_{Q_i^{\prime}}g_u(y)dy\Big)}^{\frac{1}{p}} \\& \leq C {\Big(\sum_{i=0}^{k}r_i^{\Delta}\int_{Q_i^{\prime}}g_u(y)dy\Big)}^{\frac{1}{p}},\end{aligned}\end{equation}
where $\Delta=p\delta-n-\varepsilon(p-1).$
Remember that, for the quasihyberpolic geodesic $\gamma$, which connects $x_0$ to $x$, it holds that, for any point $w$ in $\gamma$, it holds that \[{|x-w|}^{\frac{1}{\beta}}\leq C d(w,\partial \Omega),\]
which comes from the $\frac{1}{\beta}$-John property of the $\beta$-H\"older domain $\Omega$ (see \cite{KOT2002}). Furthermore, by the property of the Whitney-type decomposition, for any $w\in Q_i^{\prime}$, $\rho(w)\asymp r_i$. Thus, for any $0\leq i\leq k-1$ and any  $y\in Q_i^{\prime}$,
\[|x-y|\leq |x-\xi_i|+|\xi_i-y|\leq C{\rho(\xi_i)}^{\beta}+Cr_i\leq Cr_i^{\beta},\]
for some $\xi_i\in \gamma\cap Q_i^{\prime}$.
Assume that $\kappa<b$. Then, with respect to any $y\in Q_i^{\prime}$($0\leq i\leq k-1$), \[r_i^{\kappa-b}\leq C{|x-y|}^{\frac{\kappa-b}{\beta}}.\]
Thus, we arrive at that
\[r_i^\kappa\leq C{\rho(y)}^b{|x-y|}^{\frac{\kappa-b}{\beta}}.\]
For any $y\in Q_i\cap(2^{j+1}Q_k^{\prime}\backslash 2^jQ_k^{\prime}),$ \[|x-y|\asymp 2^j r_k,\] as to $0\leq j \leq |\log_2r_k|+1$.
As $d(\Omega)=1$, it suffices to consider \[0\leq j \leq |\log_2r_k|+1.\]
Therefore, as to $0\leq i\leq k-1$, \[r_k^\kappa\leq {C\rho(y)}^b{(2^jr_k)}^{\frac{\kappa-b}{\beta}}.\]
For $y\in Q_k^{\prime}$, $r_k^{\kappa-b}\geq 1.$
Thus, \[r_k^\kappa\leq Cr_k^{\kappa-b}\rho(y)^b\leq C r_k^{\frac{\kappa-b}{\beta}}\rho(y)^b.\]
The condition of the theorem tells us that $\Delta<b$. Thus, we may assume that $\kappa=\Delta$. Thus, by \eqref{eq2t} and the fact that there is no point in $\rr^n$ belongs to more than $C$ cubes with respect to $Q_0^{\prime}$,$Q_1^{\prime}$,...,$Q_k^{\prime}$, where the constant $C$ is independent of the choice of $x$, we obtain that
\[\begin{aligned}C&\leq \sum_{i=0}^kr_i^{\Delta}\int_{Q_i^{\prime}}g_u(y)dy \\&\leq C r_k^{\frac{\kappa-b}{\beta}}\int_{Q_k}{\rho(y)}^bg_u(y)dy+C\sum_{j=0}^{|\log_2r_k|+1}{(2^jr_k)}^{\frac{\kappa-b}{\beta}} \int_{\Omega\cap(2^{j+1}Q_k\backslash2^jQ_k)}g_u(y){\rho(y)}^bdy \\&\leq C\sum_{l=0}^{|\log_2r_k|+2}{(2^lr_k)}^{\frac{\kappa-b}{\beta}} \int_{\Omega\cap2^lQ_k}g_u(y){\rho(y)}^bdy \end{aligned}\]
By a argument similar with \eqref{eq3t}, for $t>0$, there is an $l\in[0,|\log_2r_k|+2]$ such that \[{(2^lr_k)}^t\leq C{(2^lr_k)}^{\frac{\Delta-b}{\beta}}\int_{\Omega\cap2^lQ_k}g_u(y)\rho(y)^bdy.\]
In other words, there is an $R_x>r_k\geq C\rho(x)$ such that \[CR_x^{-\frac{\Delta-b}{\beta}+t}\leq \int_{B(x,R_x)\cap\Omega}g_u(y){\rho(y)}^bdy.\]
For $y\in 5B(x,R_x)\cap\Omega$, $\rho(y)\leq \rho(x)+5R_x\leq CR_x.$ As $a\geq 0$, ${\rho(y)}^a\leq CR_x^a.$ By a similar process as in \eqref{qqq}, we obtain that \[\int_{\mathcal{C}}\rho(x)^adx\leq C{\Big(\int_{\Omega}g_u(y){\rho(y)}^bdy\Big)}^{\frac{q}{p}}.\]
As to $\mathcal{D}$, we can cover it up to a set with Lebesgue measure zero by a sequence of cubes $Q_1$,$Q_2$,$Q_3$,... that belongs to $\mathcal{V}_m(\Omega)$ with some point $x_i\in Q_i\cap\mathcal{D}$ for any $i\geq 1$. Thus, $u_{Q_i}<\frac{1}{2}.$ Therefore, by Lemma~\ref{ddd},  \[\begin{aligned}\int_{\mathcal{D}}{\rho(y)}^ady &\leq \sum_{i=0}^{\infty}\int_{Q_i\cap A}{\rho(y)}^ady \\& \leq C\sum_{i=0}^{\infty} {\ell (Q_i)}^a|Q_i\cap A| \\& \leq \sum_{i=0}^{\infty} C{\ell(Q_i)}^a\int_{Q_i}{|u-u_{Q_i}|}^q \\& \leq C\sum_{i=1}^{\infty}{\ell(Q_i)}^{a+n+\frac{q}{p}(p\delta-n)}{\Big(\int_{Q_i}g_u\Big)}^{\frac{q}{p}} \\& \leq C {(\int_{\Omega}g_u\rho^{[a+n+\frac{q}{p}(p\delta-n)]{\frac{p}{q}}})}^{\frac{q}{p}}
\\& \leq C {\Big(\int_{\Omega}g_u\rho^{b}\Big)}^{\frac{q}{p}}.\end{aligned}\]
Combining the above two inequalities, we finish the proof.

\end{proof}

Then we have an analogue of Theorem~\ref{4.2t}.
\begin{thm}\label{t4.2}
	Let $\Omega$ be a $\beta$-H\"older domain ($0<\beta\leq 1$). Assume that $0<\delta<1$, $\tau>0$, $1\leq p\leq q$, $a\geq 0$ and $p\delta-n<b<(a+n)\beta\frac{p}{q}+p\delta-n.$
	Then for any ball $B\subset\subset\Omega$, there is a constant $C$, such that, for any $u\in C(\Omega)$ with $u|_B=0$, \[{\Big(\int_\Omega{|u(x)|}^q{\rho(x)}^adx\Big)}^{\frac{1}{q}}\leq C{\Big(\int_\Omega g_u^\tau(x){\rho(x)}^bdx\Big)}^{\frac{1}{p}}.\]
\end{thm}
\begin{proof}[Proof of Theorem~\ref{t4.2}]
The conclusion follows directly from Theorem~\ref{4.1l} and Theorem~\ref{1.2t}.
\end{proof}
\begin{proof}[Proof of Theorem~\ref{t4.4}]
It follows from Theorem~\ref{thm:capactity characterization} and Theorem~\ref{1.2t}.
\end{proof}

\medskip
\textbf{Acknowledgements}. The author would like to express his gratitude to Prof.~Chang-Yu Guo for his interest in this work and for many useful discussions during the preparation of this work.
He is also willing to thank his supervisor Prof.~Jin-Song Liu for his comments and for many  thoughtful suggestions.


\begin{thebibliography}{10}

\bibitem{B2011}	
H. Brezis, \emph{Functional analysis, Sobolev spaces and partial differential equations}. Universitext. Springer, New York, 2011.	
	
	
\bibitem{BK1995}
S. Buckley and P. Koskela, \emph{Sobolev-Poincar\'e implies John}, Math. Res. Lett., \textbf{2}(5) (1995) 577-593.

\bibitem{BK1998}
S. Buckley and P. Koskela, \emph{New Poincare inequalities from old}, Ann. Acad. Sci. Fenn. Math. \textbf{23} (1998), no. 1, 251-260.

\bibitem{DD2022}
I. Drelichman and R.G. Dur\'an, \emph{The Bourgain-Br\'ezis-Mironescu formula in arbitrary bounded domains}, Proc. Amer. Math. Soc. \textbf{150} (2022), no. 2, 701-708.

\bibitem{DIV2015}
B. Dyda, L. Ihnatsyeva and A. V. Vahakangas, \emph{On improved Sobolev-Poincare inequalites}, Ark. mat., \textbf{54} (2016) 437-454.

\bibitem{G2015}
C.-Y. Guo,  \emph{Sharp capacity estimates in $s$-John domains}, Potential Anal. \textbf{43} (2015), no. 2, 277-288.

\bibitem{G2017}
C.-Y. Guo,
\emph{Fractional Sobolev Poincar\'e inequalities in irregular domains}, Chin. Ann. Math., \textbf{38B}(3) (2017) 839-856.

\bibitem{GK14}
C.-Y. Guo and P. Koskela, \emph{Generalized John disks}, Cent. Eur. J. Math. \textbf{12} (2014), no. 2, 349-361.

\bibitem{HK1998}
P. Hajlasz and P. Koskela,
\emph{Isoperimetric inequalities and imbedding theorems in irregular domains}, J. London Math. Soc., \textbf{58}(2) (1998) 425-450.

\bibitem{HK2000}
P. Hajlasz and P. Koskela, \emph{Sobolev met Poincar\'e}, Mem. Amer. Math. Soc., \textbf{145}(688) (2000).

\bibitem{HSV2013}
R. Hurri-Syrjanen and A. V. Vahakangas, \emph{On frctional Poincar\'e inequalities}, J. Anal. Math., \textbf{120} (2013) 85-104.

\bibitem{JK2013}
R. Jiang and A. Kauranen, \emph{A note on ``Quasihyperbolic boundary conditions and Poincar\'e domains''}, Math. Ann. \textbf{357} (2013), no. 3, 1199-1204.

\bibitem{J1961}
F. John, \emph{Rotation and strain}, Comm. Pure Appl. Math., \textbf{14} (1961) 391-413.

\bibitem{KM2000}
T. Kilpelainen and J. Maly, \emph{Sobolev inequalities on sets with irregular boundaries}, Z. Anal.  Anwendungen, \textbf{19}(2) (2000) 369-380.

\bibitem{KOT2002}
P. Koskela J. Onninen and J. T. Tyson, \emph{Quasihyperbolic boundary conditions and Poincar\'e domains}, Math. Ann., \textbf{323}(4) (2002) 811-830.


\bibitem{MS1979}
O. Martin and J. Sarvas, \emph{Injectivity theorems in plane and space}, Ann. Acad. Sci. Fenn. Ser. A I Math., \textbf{4}(2) (1979) 383-401.

\bibitem{M1960}
V. G. MAZ'YA, \emph{Classes of domains and imbedding theorems for function spaces}, Dokl. Akad. Nauk SSSR 133 (1960) 527-530 (in Russian); So iet Math. Dokl. 1 (1960) 882-885 (in English).

\bibitem{M1973}
V. G. MAZ'YA, \emph{On certain integral inequalities for functions of many variables}, Problems of mathematical analysis 3 (Leningrad University, Leningrad, 1972) 33-68 (in Russian); J. So iet Math. 1 (1973) 205-234 (in English).

\bibitem{M2011}
V.G. Maz'ya, \emph{Sobolev spaces with applications to elliptic partial differential equations}. Second, revised and augmented edition. Grundlehren der mathematischen Wissenschaften [Fundamental Principles of Mathematical Sciences], 342. Springer, Heidelberg, 2011.


\bibitem{SS1990}
W. Smith and D. A. Stegenga, \emph{H\"older domains and Poincar\'e domains}, Trans. Amer. Math. Soc., \textbf{319} (1990) 67-100.

\bibitem{S1970}
E. M. Stein, \emph{Singular integrals and differentiability properties of functions}, Princeton University Press, Princeton, N.J., (1970) Princeton Mathematical Series, No.30.







\end{thebibliography}
\end{document}